\documentclass[preprint,12pt]{elsarticle}

\usepackage[utf8]{inputenc}
\usepackage{amsfonts}
\usepackage{amssymb}
\usepackage{amsthm}
\usepackage{lineno}
\usepackage{blkarray} %for matrices that use blockarray

  \usepackage{paralist}
  \usepackage{hyphenat}
  \usepackage{graphics} 

\usepackage{graphicx}

\usepackage{hyperref}
\hypersetup{
    colorlinks=true,
    linkcolor=blue,
    filecolor=magenta,      
    urlcolor=cyan,
}

\usepackage{amsmath}
\usepackage{centernot}

\usepackage[table,xcdraw]{xcolor}
\usepackage{subcaption}

\usepackage[normalem]{ulem}

\newtheorem{theorem}{Theorem}[section]
\newtheorem{notation}[theorem]{Notation}
\newtheorem{corollary}[theorem]{Corollary}
\newtheorem{lemma}[theorem]{Lemma}
\newtheorem{proposition}[theorem]{Proposition}

\theoremstyle{definition}
\newtheorem{definition}[theorem]{Definition}
\newtheorem{remark}[theorem]{Remark}
\newtheorem{example}[theorem]{Example} 
\newtheorem{question}[theorem]{Question}

\newcommand{\supp}{\operatorname{supp}}
\newcommand{\F}{\mathbb{F}}
\newcommand{\T}{\mathcal{T}}

\hypersetup{final}
\begin{document}

\begin{frontmatter}



\title{Information hiding using matroid theory}

\author[1]{Ragnar Freij-Hollanti\fnref{a}}
%\ead{ragnar.freij@aalto.fi}
\author[2]{Olga Kuznetsova\fnref{b}}
%\ead{olga.kuznetsova@aalto.fi}
\address[a]{Department of Mathematics and Systems Analysis, Aalto University, ragnar.freij@aalto.fi}
\address[b]{Department of Mathematics and Systems Analysis, Aalto University, olga.kuznetsova@aalto.fi}


\begin{abstract}
Inspired by problems in Private Information Retrieval, we consider the setting where two users need to establish a communication protocol to transmit a secret without revealing it to external observers. This is a question of how large a linear code can be, when it is required to agree with a prescribed code on a collection of coordinate sets. We show how the efficiency of such a protocol is determined by the derived matroid of the underlying linear communication code. Furthermore, we provide several sufficient combinatorial conditions for when no secret transmission is possible. 

\end{abstract}


\end{frontmatter}

%%%%%%%%%%%%%%%%%%%%%%%%%%%%%%%%%%%%%%%%%%%%%%%%%%%%%%%
%%Introduction
%%%%%%%%%%%%%%%%%%%%%%%%%%%%%%%%%%%%%%%%%%%%%%%%%%%%%%%
\section{Introduction}
\label{sec1: intro}
The constuctions in this paper were originally motivated by problems in Private Information Retreival (PIR)~\cite{TGKFHR}. Rather than describing the specific setting that motivated our problem, we will give a generalized communication setup.

Consider a pair $(Q,Q')$ of linear codes of length~$n$ over a field $\mathbb{F}$ such that $Q\subseteq Q'\subseteq\F^n$.
Two parties, Alice and Bob, have agreed to use the pair $(Q,Q')$ for communication. Their communication happens in two different states with high-entropy and low-entropy messages, respectively. The reason for doing so might, for example, be that the channel quality varies over time, or (as in the case of PIR) that different information is requested to avoid surveillance. 
For simplicity, we may assume that one message between Alice and Bob corresponds to one codeword $x\in \F^n$, and that each symbol of $x$ is sent over a different link.  Depending on the application, these links may be servers, routers, or even wires. Without loss of generality, we label these links by the integers $1,\dots, n$. In the high-entropy communication state, the message is a codeword drawn from $Q'$, and in the low-entropy state it is drawn from~$Q$.

A collection $I$ of observers each control a different subset $T_\alpha\subseteq [n]$, $\alpha\in I$ of the links. These subsets need not be disjoint. Depending on the application, these observers can also be spies or computational clusters. The pair $(Q,Q')$ should now be designed so that none of the observers can tell  whether a message is sent in the high-entropy or low-entropy communication state. In other words, we insist that the projections $Q|_{T_\alpha}$ and $Q'|_{T_\alpha}$ are the same for every $\alpha\in I$. 

\begin{definition}
We call the collection $\T=\{T_\alpha : \alpha\in I\}$ a {\em collusion pattern} and its elements $T$ are \emph{colluding sets}.
\end{definition}

 Notice that, if $Q|_T=Q'|_T$, then it immediately follows that $Q|_{T'}=Q'|_{T'}$ for all $T'\subseteq T$, as $Q|_{T'}$ is a projection of $Q|_T$. We can thus assume without loss of generality that  collusion patterns are closed under inclusion:
\[T \in \mathcal{T} \quad \implies T' \in \mathcal{T} \quad \forall   T' \subset T.\] 
In other words, $\mathcal{T}$ is an abstract simplicial complex and we write 
$\mathcal{T}=\langle T_1,T_2, \cdots, T_r\rangle$, where $T_\alpha$ are the maximal colluding sets and $|I|=r$. Notice that we assume that the observers do not share information between them, so the collusion pattern is not necessarily closed under unions.
The union-closed case is just the case where the collusion pattern is a disjoint union of maximal colluding sets.

\begin{definition}\label{def: lift}
	Let $\mathcal{T}$ be a collusion pattern with ground set $[n]$, and $Q\subseteq \F^n$ a linear code over a field $\F$.  
The \emph{lift} of $Q$ over the collusion pattern $\mathcal{T}$ is  
$$Q^\mathcal{T}=\{x\in \mathbb{F}^n: \exists y \in Q \quad \text{s.t. } x|_T=y|_T\quad \forall T \in \mathcal{T}\}.$$
\end{definition}

By definition, for fixed $Q$, the code $Q^\T$ is the largest code for which the conditions $Q_{|T}=Q'_{|T}$ holds for all $T\in\mathcal{T}$. We can think of the choice in which state to communicate as a secret. 

We will now illustrate this construction with an example.
\begin{example}\label{ex: basic retrieval code}
	Let $Q \subseteq \mathbb{F}_2^5$ be the code 
generated by 
	\[
	G_Q=
	\begin{bmatrix}
	1 & 0 & 0 & 1 & 0\\
	0 & 1 & 0 & 1 & 1\\
	0 & 0 & 1 & 0 & 1\\
	\end{bmatrix}
	\]
	and let $\mathcal{T}=\{\{1,2,3\}, \{3,4,5\}\}$. Then $Q^\mathcal{T}=\mathbb{F}_2^5$. To see this, take any codeword $x=(x_1,x_2,x_3,x_4,x_5)\in \mathbb{F}_2^5$. 
The codewords $y=(x_1,x_2,x_3,x_1+x_2,x_2+x_3) \in Q$ and $y'=(x_4-x_5+x_3,x_5-x_3,x_3,x_4,x_5)\in Q$ satisfy $x|_{\{1,2,3\}}=y|_{\{1,2,3\}}$ and $x|_{\{3,4,5\}}=y'|_{\{3,4,5\}}$. Therefore,  $Q^\mathcal{T}=\mathbb{F}_2^5$.
\end{example}

In the context of PIR, the secret is the identity of the downloaded file and the code $Q$ is used to send queries to the database~\cite{chor1998private, ostrovsky2007survey}. This work is also related to oblivious transfer~\citep{di2000single}, collision-resistant hashing~\cite{ishai2005sufficient}, locally decodable codes~\cite{katz2000efficiency, yekhanin2010locally} and secret sharing~\cite{shamir1979share, beimel2012share}. Motivated by several of these applications, it is natural to define the {\em secrecy rate} as\[\frac{\dim(Q^\mathcal{T}/Q)}{n}.\] In the words of our general communication setting, this is the fraction of the channel's bandwidth that can be used to send high-entropy information that will not interfere with the low-entropy signal.

Our scheme relies on the ability of the transmitter and receiver to prevent the detection of secret transmission. In this respect, it is related to the field of information hiding, which originated in antiquity and has received renewed attention since the 1990s due to the need to protect the copyright on digital content~\cite{petitcolas1999information}. Some of the known strategies for information hiding include quantize-and-replace~\cite{barton1997method, tanaka1990embedding}, low-bit modulation~\cite{swanson1997data}, additive spread spectrum~\cite{cox1996secure} and quantization index modulation~\cite{chen2001quantization}. For a general reference, we refer the reader to~\cite{cox2007digital}. 

In this paper, we focus on the relationship between the combinatorics of the communication code $Q$ and the collusion pattern $\mathcal{T}$. In particular, we are interested in conditions under which $Q=Q^\mathcal{T}$, so that no information hiding is possible.

It has been shown that if the entire ground set is an element of the collusion pattern, then Alice and Bob cannot share secret information~\cite{chor1998private}. In the current setting, this is just the trivial observation that $Q^{\langle[n]\rangle}=Q$ for any code $Q\subseteq \F^n$.
On the other hand, if the collusion pattern is disconnected, i.e.,
\[\mathcal{T}\subseteq\mathcal{T}_1\oplus\mathcal{T}_2,\] where $\T_1$ and $\T_2$ are disjoint collusion patterns 
then it is always possible to choose $Q$ such that $Q\subsetneq Q^\mathcal{T}$~\cite{TGKFHR}.

One collusion pattern of special importance is the $t$-collusion~\citep{beimel2005general}. It allows to model the situation when only the maximal size of colluding sets is known but the pattern itself is unknown.  
\begin{definition}
A \emph{$t$-collusion} pattern is the collusion pattern given by
\[\mathcal{T}=\langle T: |T|=t\rangle.\]
\end{definition}

As we will explain in Section~\ref{sec: construction of lift}, if $\mathcal{T}$ is a $t$-collusion, then $Q=Q^\mathcal{T}$ whenever $\dim Q < t$.

The starting point of this paper is the following example~\citep[Section 7.5]{SunJafar-conjecture}.

\begin{example}\label{motivating example}
Let $Q$ is an $[n,t-1]$ MDS code. Let 
\[\mathcal{T}=\langle\{i,i+1,\dots,i+t-1\}:i \in [n]\}\rangle,\]
with indices taken modulo $n$. Then $Q=Q^\mathcal{T}$. 
\end{example}

The authors of~\citep{SunJafar-conjecture} use the symmetry of this collusion pattern to derive this result. However, the symmetry is not essential for the result to hold, and in Theorem~\ref{theorem: triangular matrix conditions for basis} we show the stronger result that any $n-t+1$ of the generators constructed as in Example~\ref{motivating example} are enough to guarantee that $Q=Q^\mathcal{T}$. 

The setting in Example~\ref{motivating example} is not unique and our goal was to understand which combinatorial properties of the code $Q$ determine when a collusion pattern $\mathcal{T}$ is equivalent to $t$-collusion. In Theorem~\ref{thm:matroid invariance} we show that the matroid structure of $Q^\mathcal{T}$ is determined by the so-called derived matroid of $Q$. 
We show that all codes with the same matroid and derived matroid will have combinatorially equivalent lifts, in a sense that is made precise in Section~\ref{sec: invariance wrt to derived matroid}. This connection to the derived matroid allows us to determine several other $t$-collusion equivalent patterns, which we present in Section~\ref{sec:t-equivalence}.

%%%%%%%%%%%%%%%%%%%%%%%%%%%%%%%%%%%%%%%%%%%%%%%%%%%%%%%
%%Preliminaries
%%%%%%%%%%%%%%%%%%%%%%%%%%%%%%%%%%%%%%%%%%%%%%%%%%%%%%%
\section{Combinatorial preliminaries}
\label{sec2:combinatorial prelims}
In this section, we will review the basics of matroid theory, with a special focus on applications in coding theory. For general references on matroid theory we refer the reader to~\cite{oxley2006matroid, oxley2019derived}. 

\begin{definition}
 A \emph{matroid} $M$ is a pair $(\mathcal{C},E)$, where $E$ is the \emph{ground set} and $\mathcal{C}\subseteq 2^{E}$ is a collection of finite subsets of $E$, called \textit{circuits}, that satisfies the following axioms:

\begin{enumerate}
	\item $\emptyset \notin \mathcal{C}$
	\item If $C',C''\in \mathcal{C}$ and $C' \subseteq C''$, then $C'=C''$
	\item If $C',C''\in \mathcal{C}$, $C' \neq C''$ and $e \in C' \cap C''$, then there exists a circuit $C \in \mathcal{C}$ such that $C \subseteq (C' \cup C'') - {e}$.
\end{enumerate}   
\end{definition}

The {\em direct sum}, or {\em disjoint union} of two matroids $M=(\mathcal{C},E)$ and $M'=(\mathcal{C}',E')$ is the matroid $$M\oplus M' = (\mathcal{C}\sqcup \mathcal{C}',E\sqcup E').$$ It is straightforward to verify that this is indeed a matroid.
In this work, $E$ is a finite set, which we identify without loss of generality with $[n]:=\{1,2,\dots,n\}$. A set is \emph{dependent} if it contains a circuit and \emph{independent} otherwise. Thus circuits are the minimal dependent sets.  A matroid that can not be written as the disjoint union of two nonempty matroids is called \emph{connected}.

A maximal independent set is called a \emph{basis}. Given a basis $B$ of a matroid and an element $e$ in $[n]-B$, there exists a unique circuit contained in $B \cup \{e\}$, known as the \emph{fundamental circuit of $e$ with respect to $B$}.

An element $e$ of the ground set is called a \emph{loop} if it is a circuit, and a \emph{coloop} if it is contained in every basis.

For any $X \subseteq [n]$ we can define the \emph{rank} $r(X)$ as the size of the largest independent set contained in $X$. In particular, the rank $r=r([n])$ of the matroid of an $\mathbb{F}$-linear code $Q$ is equal to the dimension of the code.

If $X$ is a subset of $[n]$, then the \emph{closure} $\operatorname{cl}(X)$ is the set $\{e \in [n]: r(X \cup \{e\})=r(X)\}$. If $\operatorname{cl}(X)=X$, then $X$ is a \emph{flat}. Flats of rank $r-1$ are called \emph{hyperplanes}.

The \emph{dual} matroid $M^*$ is the matroid on the same ground set $[n]$ whose circuits are the sets whose complements are hyperplanes in $M$. The dual rank of a set is given by $r^*(X)=r([n]-X)+|X|-r(X)$. 

The \textit{deletion} of $M$ by $X\subseteq [n]$ is the matroid $M\backslash X=([n]-X,\mathcal{C}\backslash X)$, where $\mathcal{C}\backslash X:=\{C \in \mathcal{C}: C \subseteq [n]-X\}$. In coding theory, deletion corresponds to puncturing of a linear code by columns in $X$. The \textit{contraction} of $M$ by $X$ is the matroid $M/X=([n]-X,\mathcal{C}/X)$, where $\mathcal{C}/X:=\{S \in 2^{[n]-X}: \exists S' \subseteq X \text{ s.t. } S \cup S'\in \mathcal{C}\}$. It corresponds to shortening of a linear code by columns in $X$.

A \textit{minor} of a matroid $M$ is the matroid that is obtained from $M$ by a sequence of deletions and contractions. Such sequences are associative and the operations of contraction and deletion commute. This implies that every  minor can be written $M\backslash Y/X$ for $X,Y \subseteq E$ and $X \cap Y = \emptyset$.

In this paper we work with \emph{$\mathbb{F}$-representable} matroids over finite ground sets $[n]$. This means matroids $M$ for which there exist a matrix $G$ over $\mathbb{F}$ with $n$ columns, such that a set $C\subseteq [n]$ is a circuit in $M$ if the columns indexed by $C$ form a minimal $\mathbb{F}$-linearly dependent set in $G$. The $\mathbb{F}$-linear code $Q$ generated by such a matrix $G$ is called a \emph{representation} of $M$.

Let $M=(\mathcal{C},[n])$ be an $\mathbb{F}$-representable matroid and $Q$ a representation of $M$. For every circuit $C$ of $M$, there exists a vector in the dual code $Q^{\perp}\subseteq \mathbb{F}^n$ that is supported on $C\subseteq [n]$. Such vectors are called \emph{circuit vectors} and they are unique up to non-zero scalar multiplication. Associated to the code $Q$, this set of circuit vectors represent a representation of the \emph{derived matroid} $\delta M (Q)$, which is an $\mathbb{F}$-representable matroid on the ground set $\mathcal{C}$~\cite{oxley2019derived}. The circuits of $\delta M (Q)$ are given by the minimal linear dependencies between the circuit vectors in $Q^{\perp}$. The derived matroid depends on the representation of $M$, so the notation $\delta M(Q)$ always refers to the specific representation $Q$. It is proven in~\cite{oxley2019derived} that a matroid $M$ is connected if and only if all of its derived matroids are.

%%%%%%%%%%%%%%%%%%%%%%%%%%%%%%%%%%%%%%%%%%%%%%%%%%%%%%%
%%Construction of the lift
%%%%%%%%%%%%%%%%%%%%%%%%%%%%%%%%%%%%%%%%%%%%%%%%%%%%%%%
\section{Construction of the lift} 
\label{sec: construction of lift}
It will often be useful to consider equivalent definitions of the lift, which we will present below.

\begin{notation}\label{notation: distance}
Let $d(x,y)$ be the Hamming distance of two strings $x$ and $y$ of equal length, let $Q$ be a linear code, and let $T$ be a subset of the coordinate set of $Q$. We write:
\begin{itemize}
	\item $d_{|T}(x,y)$ to denote the Hamming distance of $x$ and $y$ restricted to a subset $T$ of coordinates, \emph{i.e.,} $$d_{|T}(x,y)=|\{i\in T|x_i \neq y_i\}|;$$
	\item $d(x,Q):=\min\limits_{y \in Q} d(x,y)$ and $d_{|T}(x,Q):=\min\limits_{ y \in Q}d_{|T}(x,y)$.
\end{itemize} 
\end{notation}

\begin{lemma}\label{def: lift via dual code}
Let $Q\subseteq \F^n$ be a linear code with dual code $Q^\perp$. Let $x\in\F^n$ and let $T\subseteq [n]$. Then $d_{|T}(x,Q)=0$ holds if and only if $x\cdot v=0$ for all $v\in Q^\perp$ with $\supp(v)\subseteq T$.
\end{lemma}
\begin{proof}

We denote 
$$Q^\perp(T) := \{v\in Q^\perp : \supp v\subseteq T\},$$ 
and note that this is a linear subspace of $Q^\perp$. First, assume that there exist\sout{s} $v\in Q^\perp(T)$, such that $x\cdot v\neq 0$, and $x'\in Q$ with $d_{|T}(x',x)=0$. Then $x'\cdot v=x\cdot v\neq 0$, which is a contradiction, and thus $d_{|T}(x,Q)\neq 0$.

Conversely, assume that  $x\cdot v=0$ for all $v\in Q^\perp(T)$. We will prove by induction on $n-|T|$ that there exists $y\in Q$ such that $y_{|T}=x_{|T}$. This is immediate (with $y=x$) if $T=[n]$. Now if $T\neq[n]$, choose $t\notin T$, and let $T'=T\cup\{t\}$. As an induction hypothesis, assume that if $x'\cdot v=0$ for all $v\in Q^\perp(T')$, then there exists $y\in Q$ such that $y_{|T'}=x_{|T'}$. Now clearly $Q^\perp(T)\subseteq Q^\perp(T')$ and $\dim Q^\perp(T')\leq Q^\perp(T)+1$, so there is a codeword $w\in Q^\perp(T')$ such that $Q^\perp(T')=Q^\perp(T)\oplus \{c\}$. Now let $x'$ agree with $x$ outside of $t$, and get a value on $t$ such that $x'\cdot w = 0$. Then $x'\cdot v=0$ for all $v\in Q^\perp(T')$, and so by induction there is $y\in Q$ with $y_{T'}=x'_{T'}$, wherefore $y_{T}=x'_{T}=x_T$.
\end{proof}

\begin{proposition}
Definition~\ref{def: lift} of the lift can be rewritten as
\begin{enumerate}
    \item $Q^\mathcal{T}=\{x\in \mathbb{F}^n: d_{|T}(x,Q)=0\quad \forall T \in \mathcal{T}\}$
    \item $Q^\mathcal{T}=\{x\in \mathbb{F}^n: x\cdot v=0\ \text{for all } v\in Q^\perp(T) \ \text{and} \ T \in \mathcal{T}\}$. 
\end{enumerate}
\end{proposition}
\begin{proof}
The first equality is just Definition~\ref{def: lift} restated using Notation~\ref{notation: distance}. The second one follows from Lemma~\ref{def: lift via dual code}.
\end{proof}

\begin{lemma}\label{lemma: union and intersection of collusion patterns}
Let $\mathcal{S}$ and $\mathcal{T}$ be two collusion patterns on ground set $[n]$, and let $Q\in\F^n$ be a linear code. Then \begin{enumerate}\label{lm:collops}
    \item $Q^{\mathcal{S}\cup\mathcal{T}}=Q^\mathcal{S}\cap Q^\mathcal{T}$.
    \item $Q^{\mathcal{S}\cap\mathcal{T}}=(Q^\mathcal{S})^\mathcal{T}$.
\end{enumerate}
\end{lemma}
\begin{proof}
By Lemma \ref{def: lift via dual code} we have that $Q^{\mathcal{S}\cup\mathcal{T}}$ consists of those $x\in \F^n$ that are orthogonal to all $v\in Q^\perp$ with $\supp(v)\in \mathcal{S}\cup\mathcal{T}$, which is equivalent to that $x$ is orthogonal both to those $v\in Q^\perp$ that have $\supp(v)\in \mathcal{S}$ and those that have $\supp(v)\in \mathcal{T}$, so $x\in Q^\mathcal{S}\cap Q^\mathcal{T}$.

Using Lemma \ref{def: lift via dual code} twice, we have that 
\begin{align*}
(Q^{\mathcal{S}})^{\mathcal{T}}&=\{x\in\F^n : x\cdot v=0\mbox{ for all } v\in (Q^\mathcal{S})^{\perp} \mbox{ with } \supp v\in \mathcal{T}\}\\
&= \{x\in\F^n : x\cdot v=0\mbox{ for all } v\in Q^\perp \mbox{ with } \supp v\in \mathcal{S}\cap\mathcal{T}\}\\
&= Q^{\mathcal{S}\cap\mathcal{T}}.
\end{align*}
\end{proof}
\begin{corollary}\label{cor:connected}
Let $\T=\T_1\sqcup \T_2$ is disconnected, with the vertex sets of $\T_1$ and $\T_2$ disjoint, and let $Q_i$ denote the restriction of $Q$ to the vertex set of $\T_i$. Then $Q^\mathcal{T}=Q_1^{\T_1}\times Q_2^{\T_2}\times \F^{[n]-V_1-V_2}$. 
\end{corollary}
\begin{proof}
Clearly by definition, $Q^{\T_i}=Q_i^{\T_i}\times \F^{[n]-V_i}$. Thus we get by Part 1 of Lemma~\ref{lemma: union and intersection of collusion patterns} that $$Q^\mathcal{T}=Q_1^{\T_1}\times \F^{[n]-V_1}\cap Q_1^{\T_2}\times \F^{[n]-V_2} = Q_1^{\T_1}\times Q_2^{\T_2}\times \F^{[n]-V_1-V_2}.$$
\end{proof}

By definition, if $\mathcal{T}$ and $\mathcal{T}'$ are two collusion patterns such that $\mathcal{T} \subseteq \mathcal{T}'$, then for any linear code $Q$, $Q^\mathcal{T} \supseteq Q^{\mathcal{T}'}$. The following lemma, which will be used many times in the paper, shows that we can restrict attention to dual vectors that are supported on the collusion pattern over which we are lifting.
\begin{lemma}\label{lemma: lift given by colluding circuits}
Let $Q$ be a code with matroid $M(\mathcal{C},[n])$, and let $\mathcal{T}$ be a collusion pattern on ground set $[n]$. Then $$Q^{\mathcal{T}}=Q^{\mathcal{T}\cap\mathcal{C}}.$$ 
\end{lemma}
\begin{proof}
The linear conditions are generated by the circuit vectors, so \begin{align*}Q&=\{x\in\F^n : x\cdot v=0 \mbox{ for all } v\in Q^\perp \mbox{ with }\supp(v)\in\mathcal{C}\}\\
& =\bigcap_{C\in \mathcal{C}}\{x\in\F^n : x\cdot v=0 \mbox{ for } v\in Q^\perp \mbox{ with }\supp(v)=C\}\\
&= \bigcap_{C\in \mathcal{C}}\{x\in\F^n : x_C\in Q_C\}=Q^\mathcal{C}.\end{align*} By Lemma~\ref{lm:collops}, we thus have $Q^{\mathcal{T}\cap\mathcal{C}}=(Q^\mathcal{C})^\mathcal{T}=Q^\mathcal{T}$.
\end{proof}

\begin{corollary}\label{cor: elementary properties}
Let $Q$ be a code with matroid $M(\mathcal{C},[n])$, and let $\mathcal{T}$ be a collusion pattern on ground set $[n]$. Then
\begin{enumerate}
    \item if $\mathcal{C} \subseteq \mathcal{T}$, then $Q = Q^\mathcal{T}$;
    \item $Q^\mathcal{T} = \mathbb{F}^n$ if and only if $\mathcal{C} \cap \mathcal{T}=\emptyset$; \label{lift is field}
    \item If $Q$ is an $[n,k]$ MDS code and for all $T$, $|T|\leq k$, then $Q^\mathcal{T}=\mathbb{F}^n$.
\end{enumerate}
\end{corollary}

\begin{definition}
We call the circuit vectors supported on $\T \cap \mathcal{C}$ \emph{observed} circuit vectors.
\end{definition}

Let $Q$ be a linear code over $\mathbb{F}$ and $\mathcal{T}$ a collusion pattern. We will now show how to construct the lift $Q^\mathcal{T}$ by identifying a basis of the dual $(Q^\mathcal{T})^\perp$. By Lemma~\ref{lemma: lift given by colluding circuits}, we can restrict our attention to the observed circuit vectors.

Fix a generator matrix $G_Q$ for $Q$ and take its row-reduced restriction on $T$, which we denote by $(G_Q)_{|T}$. If $(G_Q)_{|T}\equiv I_{|T|}$, then $Q^T=\mathbb{F}^{|T|}$. On the other hand, if $(G_Q)_{|T}\not\equiv I_{|T|}$, then there exists a vector in $Q^{\perp}$ supported on $T$. In particular, if $(G_Q)_{|T}\equiv(I_{|T|-1}|v)$, where $v$ is a column with non-zero entries, then $T$ corresponds to the support of a circuit vector, and the exact coordinates of this circuit vector are given by $v$.

We are now ready to describe an algorithm for constructing $Q^\mathcal{T}$.

\begin{proposition}\label{prop: L is lift}
	Let $Q$ be a code over $\mathbb{F}$ and $\mathcal{T}$ a collusion pattern. The lift $Q^\mathcal{T}$ can be constructed using the following algorithm:
	\begin{enumerate}
	\item Identify the set $\mathcal{V}$ of observed circuit vectors using restrictions of $G_Q$ on each colluding set;
	\item Form a matrix $H$ whose rows are the circuit vectors in $\mathcal{V}$;
	\item Row reduce and bring $H$ to the systematic form without changing the labels of the columns (although their order might change).\label{step: index}
	\item Produce the parity check matrix $G$ of $H$ and permute its columns, so that their order corresponds to the index in Step~\ref{step: index}. The matrix $G$ is a generator matrix of $Q^\mathcal{T}$.
\end{enumerate}
\end{proposition}

\begin{proof}
	The matrix $G$ is constructed to be a generator matrix of some vector space $V$, so it is unique up to elementary row operations. By construction, $V$ is generated by the circuit vectors in $\mathcal{C}\cap\mathcal{T}$, so, by Lemma~\ref{lemma: lift given by colluding circuits}, we have $V=Q^\mathcal{T}$.
\end{proof}

We will now illustrate the algorithm with an example.

\begin{example}\label{example: construction of lift}
	Let $Q$ be the Reed-Solomon $RS_7[7,3]$ code whose canonical generator matrix is
	\[
	G_Q = 
	\begin{bmatrix}
	1	&1	&1	&1 &1 &1 &1\\
	0	&1	&2	&3 &4 &5 &6\\
	0	&1	&4	&2 &2 &4 &1\\
	\end{bmatrix}
	.\]

	Let the maximal elements of $\mathcal{T}$ be $\{\{1,2,3,4\},\{2,3,5,6\},\{4,5,6,7\}\}$.
	Then the collection of observed circuit vectors is 
	\[\mathcal{V}=\{x_1+4x_2+3x_3=x_4,\quad x_2+5x_3+2x_5=x_6, \quad x_4+4x_5+3x_6=x_7\},\]
	which corresponds to matrix
	\[
	H=
	\begin{blockarray}{ccccccc}
	i & ii & iii & iv & v & vi & vii \\
	\begin{block}{[ccccccc]}
	1	&4	&3	&6 &0 &0 &0\\
	0	&1	&5	&0 &2 &6 &0\\
	0	&0	&0	&1 &4 &3 &6\\
	\end{block}
	\end{blockarray}
	\equiv  
	\begin{blockarray}{ccccccc}
	i & ii & iv & iii & v & vi & vii \\
	\begin{block}{[ccccccc]}
	1	&0	&0	&4 &3 &0 &6\\
	0	&1	&0	&5 &2 &6 &0\\
	0	&0	&1	&0 &4 &3 &6\\
	\end{block}
	\end{blockarray}
	.\]
This gives a generator matrix for $Q^\mathcal{T}$:
	\[
	G = 
	\begin{blockarray}{ccccccc}
	i & ii & iv & iii & v & vi & vii \\
	\begin{block}{[ccccccc]}
	3	&2	&0	&1 &0 &0 &0\\
	4	&5	&3	&0 &1 &0 &0\\
	0	&1	&4	&0 &0 &1 &0\\
	1	&0	&1	&0 &0 &0 &1\\
	\end{block}
	\end{blockarray}
	\equiv \begin{blockarray}{ccccccc}
	i & ii & iii & iv & v & vi & vii \\
	\begin{block}{[ccccccc]}
	3	&2	&1	&0 &0 &0 &0\\
	4	&5	&0	&3 &1 &0 &0\\
	0	&1	&0	&4 &0 &1 &0\\
	1	&0	&0	&1 &0 &0 &1\\
	\end{block}
	\end{blockarray}.\]
\end{example}

%%%%%%%%%%%%%%%%%%%%%%%%%%%%%%%%%%%%%%%%%%%%%%%%%%%%%%%
%%Derived matroid of the code
%%%%%%%%%%%%%%%%%%%%%%%%%%%%%%%%%%%%%%%%%%%%%%%%%%%%%%%

\section{Invariance with respect to derived matroid} 
\label{sec: invariance wrt to derived matroid}

Let $Q$ be a linear code of length $n$ and dimension $t-1$ over $\F$ and $\T$ a collusion pattern. By Corollary \ref{cor:connected}, the lifted code $Q^\T$ is the direct sum of the lifts over all the connected components of $\T$, together with a copy of $\F$ for each that is not contained in any colluding set. Moreover, by Lemma~\ref{lemma: lift given by colluding circuits}, the collusion pattern $\T$ can be replaced by $\T\cap \mathcal{C}$, the largest elements of which have cardinality one larger than the dimension of the code $Q$.
Thus, in order to understand lifts over arbitrary collusion patterns, we will henceforth restrict attention to patterns such that
\begin{enumerate}
    \item $\T$ is connected, i.e. for every maximal colluding set $T_i$ in $\T$, there is another maximal $T_j$ in $\T$ such that $T_i \cap T_j \neq \emptyset$;
    \item for every $i \in [n]$, there exists a colluding set $T$ such that $\{i\}\subsetneq T$;
    \item every maximal colluding set $T$ in $\T$ has cardinality at most $t$;

\end{enumerate}

\begin{proposition}  \label{prop: no collusion equivalence}
Let $\mathcal{T}$ and $\mathcal{T}'$ be two different collusion patterns. Then there exists a code $Q$ such that
\[Q^\mathcal{T}\neq Q^{\mathcal{T}'}.\]

\end{proposition}

\begin{proof}
    Take $T\in \mathcal{T}-\mathcal{T}'$. Choose $Q\subsetneq \mathbb{F}_q^n$ so that $T$ is the only circuit in the matroid $M(Q)$. Then, by Corollary~\ref{cor: elementary properties},  $Q^{\mathcal{T}'}=\mathbb{F}_q^n$ but $Q^{\mathcal{T}}=Q$.
\end{proof}

\begin{corollary}
Let $\mathcal{T}=\{T\subseteq E:|T|\leq t\}$. There does not exist a collusion pattern $\mathcal{T}'\subsetneq \mathcal{T}$ such that
\[Q^{\mathcal{T}'}=Q^{\mathcal{T}}\quad \forall Q.\]
\end{corollary}

Furthermore, it is not enough to consider the matroid of $Q$ as we illustrate in Example~\ref{counterexample}.

\begin{example}
\label{counterexample}
	Let $Q_1$ and $Q_2$ be two $[6,3]$ MDS codes over $\mathbb{F}_7$ as follows. A generator matrix for the dual code $Q_1^\perp$ is 
	\[G_{Q_1^\perp}=
	\begin{bmatrix}
	1 & 2 & 1 & 5 & 0 & 0\\
	1 & 5 & 0 & 0 & 5 & 1\\
	0 & 0 & 5 & 1 & 2 & 1
	\end{bmatrix}
	\]
	and a generator matrix for the dual code $Q_2^\perp$ is
	\[G_{Q_2^\perp}=
	\begin{bmatrix}
	1 & 1 & 1 & 1 & 0 & 0\\
	0 & 0 & 1 & 1 & 1 & 1\\
	1 & 2 & 3 & 4 & 5 & 6
	\end{bmatrix}
	.\]
	Hence, $M(Q_1)=M(Q_2)$. Assume the collusion pattern is
	\[\mathcal{T}=\{1234,1256,3456\}.\] 
	
	In $Q_1$, the unique dual codewords supported on the facets of $\mathcal{T}$ are 
	\[\{(1,2,1,5,0,0), (1,5,0,0,5,1), (0,0,5,1,2,1)\}\] 
	and they form a basis of $Q_1^{\perp}$, so $Q_1=Q_1^\mathcal{T}$. Observe that uniqueness follows from the fact that the facets of $\mathcal{T}$ are circuits in $M(Q_1)$.
	
	On the other hand, in $Q_2$, the corresponding codewords  are
	\[\{(1,1,1,1,0,0), (0,0,1,1,1,1), (1,1,0,0,6,6)\},\] 
	where 
	\[(1,1,0,0,6,6)=(1,1,1,1,0,0)-(0,0,1,1,1,1),\]
	so $Q_2\subsetneq Q_2^\mathcal{T}$ 
	and uniqueness again follows from the fact that the facets are circuits in $M(Q_2)$. 
\end{example}

At the same time, Example~\ref{counterexample} suggests that the matroid of the lift $Q^\T$ is determined by the derived matroid of $Q$. We will make this precise in Theorem~\ref{thm:matroid invariance}.

\begin{definition}
Let $Q_1$ and $Q_2$ be two different linear codes such that $M(Q_1)=M(Q_2)$.
A collusion pattern $\mathcal{T}$ \emph{separates} $Q_1$ and $Q_2$ if $M(Q_1^\mathcal{T})\neq M(Q_2^\mathcal{T})$.
\end{definition}

\begin{lemma}[Separating collusion patterns]\label{lemma: construction of sep collusion patterns}
Let $Q_1$ and $Q_2$ be two different linear codes such that $M(Q_1)=M(Q_2)$ but $\delta M(Q_1)\neq\delta M(Q_2)$. Then there exists a non-empty set of separating collusion patterns 
\[\{\mathcal{T}: M(Q_1^\mathcal{T})\neq M(Q_2^\mathcal{T})\}.\]
\end{lemma}

\begin{proof}
    Since the two codes have different derived matroids, there exists a set $\mathcal{S}$ of circuits that is independent in $\delta M(Q_1)$ but dependent in $\delta M(Q_2)$. Choosing $\mathcal{T}=\mathcal{S}$ gives a separating collusion pattern.
\end{proof}

\begin{question}\label{question: separating collusion patterns}
Let $Q_1$ and $Q_2$ be two different linear codes with $M(Q_1)=M(Q_2)$ but $\delta M(Q_1)\neq\delta M(Q_2)$.
Is there a combinatorial characterisation of the set of separating collusion patterns \[\{\mathcal{T}: M(Q_1^\mathcal{T})\neq M(Q_2^\mathcal{T})\}?\]
\end{question}
The following lemma uses the derived matroid to establish the relationship between the matroid of the code and the matroid of the lift.
\begin{lemma}\label{lemma: matroid of lift is a flat}
Let $Q$ be a code over $\mathbb{F}$ and $\mathcal{T}$ a collusion pattern. Then
\[M(Q^\mathcal{T})\cap M(Q) =\mathcal{F}(\mathcal{S}),\]
where $\mathcal{S}\subseteq \mathcal{C}\cap \mathcal{T}$ is a maximal subset of colluding sets that is independent in the derived matroid $\delta M(Q)$ and $\mathcal{F}(\mathcal{S})$ is the flat of $\mathcal{S}$ in $\delta M(Q)$.
\end{lemma}

\begin{proof}
Recall from Proposition~\ref{prop: L is lift} that $Q^\mathcal{T}$ is generated by the circuit vectors corresponding to the colluding circuits $\mathcal{C}\cap \mathcal{T}$. Therefore, every circuit vector of $M(Q^\mathcal{T})\cap M(Q)$ is an element of $\mathcal{F}(\mathcal{C}\cap \mathcal{T})$. But 
\[\mathcal{C}\cap \mathcal{T} \subseteq \mathcal{F}(\mathcal{S}),\]
so $M(Q^\mathcal{T})\cap M(Q)\subseteq\mathcal{F}(\mathcal{S})$.

On the other hand, if $C \in \mathcal{F}(\mathcal{S})$, then any codeword in $Q^\mathcal{T}$ supported on $C$ should satisfy the linear equation given by its circuit vector. Hence, $C \in M(Q^\mathcal{T})$.
\end{proof}

We are now ready to show that the operation of taking a lift is invariant with respect to the equivalence class of derived matroids.

\begin{theorem}\label{thm:matroid invariance}
	Let $Q_1$ and $Q_2$ be two different linear codes with $M(Q_1)=M(Q_2)$ 
	Then $M(Q_1^\mathcal{T})=M(Q_2^\mathcal{T})$ for any connected collusion pattern $\mathcal{T}$ if and only if $\delta M(Q_1)=\delta M(Q_2)$.
\end{theorem}
\begin{proof}
	The fact that the derived matroids need to be equal follows from Lemma~\ref{lemma: construction of sep collusion patterns}.
    
    For the other direction, let $\mathcal{S}$ be a largest independent set of $\delta M(Q_1)$ contained in $\mathcal{T}$. Let $\mathcal{F}_1(\mathcal{S})$ be the flat of $\mathcal{S}$ in $\delta M(Q_1)$. Then by Lemma~\ref{lemma: matroid of lift is a flat}, $M(Q_1^\mathcal{T})\cap M(Q_1)=\mathcal{F}_1(\mathcal{S})$. By assumption, the set $\mathcal{S}$ is also a largest independent set of $\delta M(Q_2)$ contained in $\mathcal{T}$ and $\mathcal{F}_1(\mathcal{S})=\mathcal{F}_2(\mathcal{S})$. Therefore, $M(Q_1^\mathcal{T})\cap M(Q_1)=M(Q_2^\mathcal{T})\cap M(Q_2)$.
    
    Let $C^*$ be a circuit in $M(Q_1^\mathcal{T})- M(Q_1)$. This means that $C^*$ is a dependent set in $M(Q_1)$ but every circuit of $M(Q_1)$ contained in $C^*$ is an independent set in $M(Q_1^\mathcal{T})$. Therefore, if $C \subsetneq C^*$ is a circuit of $M(Q_1)$, then $C \notin \mathcal{F}_1(\mathcal{S})$. But then $C \notin \mathcal{F}_2(\mathcal{S})$, so it is an independent set in $M(Q_2^\mathcal{T})$.

   For contradiction, assume that $C^*$ is an independent set in $M(Q_2^\mathcal{T})$. For a fixed code $Q$, denote by $v(S)$ a dual vector supported on the subset $S$ of the coordinates. By assumption, there exists a dual vector $v(C^*) \in Q_1^\perp$ such that 
    \[v(C^*)=\sum_{C \in \mathcal{C}\cap \mathcal{T}}\alpha_Cv(C).\]
    By construction, each $v(C)$ is unique up to scalar, whenever $C\in \mathcal{C}\cap \mathcal{T}$. Furthermore, $v(C^*)$ is also unique up to scalar since it becomes a circuit vector in $(Q_1^\mathcal{T})^\perp$. Without loss of generality, we can assume that all $\alpha_c$ are non-zero, i.e., $\mathcal{C}\cap \mathcal{T}=\mathcal{S}$.

    Among the circuits $C \subsetneq C^*$ in $M(Q_1)$, choose a collection $\Sigma$ such that  $(\mathcal{C}\cap \mathcal{T}) \cup \Sigma$ is an independent set of maximal rank in $\delta M(Q_1)$. Consider a dual vector of the form
    \[v=\beta v(C^*)+\sum_{C \in \Sigma}\gamma_Cv(C).\]
    If not all $\beta$ and $\gamma_c$ are zero, then $v$ is a non-zero vector. 
 Assume that $v$ is chosen to have minimal support and let the set $X$ be its support. Since $v$ is a linear combination of circuit vectors, $X$ is a dependent set in $M(Q_1)$. Furthermore, it is a circuit because if $C'\subsetneq X \subsetneq C^*$ is a circuit in $M(Q_1)$, then $(\mathcal{C}\cap \mathcal{T}) \cup \Sigma \cup \{C'\}$ is an independent set in $\delta M(Q_1)$, but this contradicts the maximality of $(\mathcal{C}\cap \mathcal{T}) \cup \Sigma$. Therefore, $(\mathcal{C}\cap \mathcal{T}) \cup \Sigma \cup \{X\}$ is a dependent set in $\delta M(Q_1)$, so it is also a dependent set in $\delta M(Q_2)$. This is a contradiction, because $C^*$ is an independent set in $M(Q_2^\mathcal{T})$, so there is no dual vector $v(C^*) \in Q_2^\perp$ such that 
    \[v(C^*)=\sum_{C \in \mathcal{C}\cap \mathcal{T}}\alpha'_Cv(C).\]
\end{proof}

\begin{question}\label{q:generic}
The property that two codes have the same matroid and the same derived matroid is an equivalence relation. The realization space of an $\mathbb{F}$-representable matroid $M$ is thus partitioned according to the derived matroids. Is it possible to characterise this partition, and is there a natural way in which one of the parts corresponds to a
``generic'' derived matroid of $M$?
\end{question}
We can think of at least two different meanings of genericity in Question~\ref{q:generic}. One interpretation is that one of the parts would have higher dimension than all the others, as varieties over $\F$. Another interpretation is that there would be a natural partial order of the derived matroids with a unique maximal element.  A related question with a more probabilistic flavour is the following:

\begin{question}
Fix a $\mathbb{F}$-representable matroid $M$ and choose two representations $Q_1$ and $Q_2$ uniformly at random. What is the probability that they will have the same derived matroid? 
\end{question}

In~\cite{jurrius2015derived}, a different but related notion of \emph{derived code} was defined. For this definition, they conjecture a positive answer to Question~\ref{q:generic}, with an order-theoretic interpretation of genericity.

\section{Collusion patterns equivalent to \texorpdfstring{$t$}--collusion}
\label{sec:t-equivalence}

Recall from Corollary~\ref{cor: elementary properties} that $Q^\T=Q$ whenever $\T$ is a $t$-collusion pattern and $\dim Q < t$. On the other hand, Example~\ref{motivating example} shows that if $Q$ is an $[n,t-1]$ MDS code, it is sufficient to take the cyclical collusion pattern given by 
\[\mathcal{T}=\langle\{i,i+1,\dots,i+t-1\}:i \in [n]\}\rangle\]
to ensure that $Q^\T=Q$. In this section we will study how much $t$-collusion can be reduced to maintain $Q^\T=Q$ for a given $Q$.

\begin{definition}
Let $Q$ be a linear code over $\F$ of dimension $t-1$ and $\T$ a proper subset of $t$-collusion. We say that $\T$ is \emph{equivalent to $t$-collusion} if $Q^\T=Q$.
\end{definition}

By Lemma~\ref{lemma: construction of sep collusion patterns}, a collusion pattern $\mathcal{T}$ is  $t$-collusion equivalent for all $\mathbb{F}$-representations of a matroid $M$ if and only if it contains a basis of every derived matroid $\delta M$ given by the $\mathbb{F}$-representations. On the other hand, Theorem~\ref{thm:matroid invariance} shows that the set of $t$-equivalent collusion patterns is invariant with respect to the derived matroid of the linear code.  

In this section we will state several sufficient conditions for $t$-equivalence. The next lemma shows that we can assume that both the matroid of the underlying code and the collusion pattern are connected.

\begin{lemma}\label{lemma: properties of derived matroid}
    \begin{enumerate}
    \item If an element $e$ is a loop in $M$, then $\delta M=U_{1,1}\oplus \delta (M\backslash e)$. If an element $e$ is a coloop in $M$, then $\delta M=\delta (M\backslash e)$. 
    \item If $\mathcal{T}$ is disconnected and $M$ is connected, then $Q \subsetneq Q^\mathcal{T}$.  
    \end{enumerate}
\end{lemma}

\begin{proof}
\begin{enumerate}
        \item This is \cite[Lemma 10]{oxley2019derived}. If an element $e$ is a loop, then it is a circuit. Then up to non-zero scalar multiplication, there exists a unique circuit vector whose support is $\{e\}$. Hence, $\{e\}$ does not belong to any circuit in $\delta M$, so it is a coloop. Therefore, $\delta M=U_{1,1}\oplus \delta (M\backslash e)$.
        
If an element $e$ is a co-loop, then it is not contained in any circuit. Hence, for every representation of $M$,
\[\delta M = \delta M(\mathcal{C},[n])=\delta M(\mathcal{C},[n]-\{e\})= \delta (M \backslash e).\]
     
        \item Let $Q$ be an $\mathbb{F}$-representation of $M$, and let $\{T_i\}_{i\in I}$ be the connected components of $\T$.   Since $M$ does not have any coloops, $\mathcal{C}\not \subseteq \mathcal{T}_i$ for all $i$. On the other hand, if $\mathcal{C}\cap \mathcal{T}_i$ is empty for some $i$, then $Q \subsetneq Q^\mathcal{T}$, so $\mathcal{T}$ is not $t$-collusion equivalent. Assume $\mathcal{C}\cap \mathcal{T}_i$ is non-empty for all $i$. Then 
\[Q^{\mathcal{T}}=\bigoplus_{i\in I} Q^{\mathcal{T}_i}\supsetneq Q.\]
\end{enumerate}
\end{proof}

\begin{remark}\label{remark: cardinality of basis of derived matroid}
Since the circuit vectors of an $[n,t-1,d]$ code $Q$ generate the dual code $Q^\perp$, the inclusion maximal independent sets of circuits have cardinality $n-t+1$. This is the same as saying that the rank of the derived matroid is $n-t+1$. 
\end{remark}

\begin{lemma}\label{lemma: outer conditions on circuit independence}
Let $M$ be an $\mathbb{F}$-representable matroid on the ground set $[n]$. Take any set of circuits $\mathcal{S}=\{C_1,C_2,\dots,C_m\}$ such that every circuit $C_i$ satisfies
\[C_i-\bigcup\limits_{C_j \in \mathcal{S}-\{C_i\} }C_j \neq \emptyset.\]
Then $\mathcal{S}$ is independent in the derived matroid of every $\mathbb{F}$-representation of $M$.
\end{lemma}

\begin{proof}
Consider the set $\mathcal{V}$ of the circuit vectors corresponding to the elements of $\mathcal{S}$.
Take any $v_i \in \mathcal{V}$. Since 
\[C_i-\bigcup\limits_{C_j \in \mathcal{S}-\{C_i\} }C_j=:\Delta \neq \emptyset,\]
the vector $v_i$ will be the only vector in $\mathcal{V}$ to have non-zero coordinates on $\Delta$. Hence, $v_i$ cannot be a linear combination of $\mathcal{V}-\{v_i\}$.
\end{proof}

The following is a useful specialization of~\cite[Lemma 3]{oxley2019derived}.

\begin{proposition}\label{prop: fundamental circuit condition}
    Let $M$ be an $\mathbb{F}$-representable matroid of rank $t-1$. Let $\mathcal{T}$ contain all fundamental circuits of some basis of $M$. Then $\mathcal{T}$ is  $t$-collusion equivalent for all $\mathbb{F}$-representations $M$.
\end{proposition}

\begin{proof}
Recall that a fundamental circuit $C(e,B)$ of $e \in [n]-B$ is the unique circuit contained in $B\cup e$. This means that $\mathcal{T}$ contains a set satisfying Lemma~\ref{lemma: outer conditions on circuit independence}. Furthermore, the existence of this set is independent of the representation of $M$. This set has cardinality $n-|B|=n-t+1$. By Remark~\ref{remark: cardinality of basis of derived matroid}, this is a basis in every derived matroid $\delta M$ and hence $\mathcal{T}$ is $t$-collusion equivalent.
\end{proof}

\begin{corollary}
Let $S \subseteq [n]$ be a set of cardinality $t-1$. A collusion pattern given by
\[\left\langle \{e_i \cup S : e_i \in [n]-S\}\right \rangle\]
is $t$-collusion equivalent for every $[n,t-1]$ MDS code.
\end{corollary}

The standard construction of a parity check matrix of a code that involves the transposition of the parity check elements is an example of the bases described in Proposition~\ref{prop: fundamental circuit condition}.
Theorem~\ref{theorem: triangular matrix conditions for basis} describes when a $t$-collusion equivalent pattern corresponds to the support of a triangular matrix. We provide examples of such collusion patterns on the ground set of $6$ element in Figure~\ref{fig:table-staircases}.

\begin{definition}
An ordered family $C_1, C_2, \cdots, C_k$ of sets is {\em triangular} if there exist elements $e_1, e_2, \cdots, e_{k}$ such that  $e_i \in C_i$ and
\[\{e_1,e_2,\cdots,e_{i-1}\}\cap C_i = \emptyset \quad \forall i >1.\]
\end{definition}

\begin{definition}
Let $C_1,\dots , C_k$ and $D$ be sets such that $D\subseteq \cup_{i=1}^k C_i$.
An element $e\in \cup_{i=1}^k C_i$ is called \emph{hanging} with respect to the pair $(D,\{C_1,\dots , C_k\})$ if $e\notin D$ and $e$ is contained in exactly one $C_{i}$. A collusion pattern $\mathcal{T}$ is called hanging if for every collection of generators $C_1,\dots , C_k, D$ with $D\subseteq \cup_{i=1}^k C_i$, there is a hanging element with respect to the pair $(D,\{C_1,\dots , C_k\})$.
\end{definition}

\begin{theorem}[Patterns defining a triangular matrix]\label{theorem: triangular matrix conditions for basis}
Let $M$ be an $\mathbb{F}$-representable matroid of rank $t-1$. Let $\mathcal{T}$ be a collusion pattern generated by $n-t+1$ circuits of $M$. Then the following characterisations of the generating set are equivalent and such $\mathcal{T}$ is $t$-collusion equivalent in every $\mathbb{F}$-representation of $M$.
\begin{enumerate}
\item \label{itm:triangular}There is an ordering $C_1, C_2, \cdots, C_{n-t+1}$ of the generators of $\T$ that is triangular.

\item \label{itm:hanging}
The  collusion pattern $\T$ is hanging. 
\end{enumerate}

\end{theorem}

\begin{figure}[!ht]
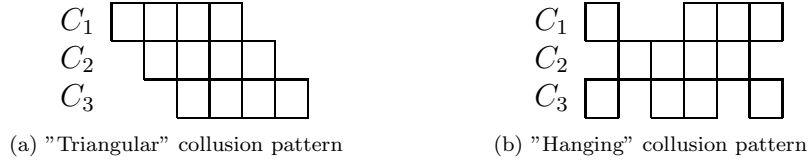

\centering
 \begin{subfigure}[t]{.45\linewidth}
 \centering
    \begin{tabular}{lll|l|l|ll}
\cline{2-5}
\multicolumn{1}{l|}{$C_1$} & \multicolumn{1}{l|}{} &  &  &  & &  \\ \cline{2-6}
$C_2$                      & \multicolumn{1}{l|}{} &  &  &  & \multicolumn{1}{l|}{} &                       \\ \cline{3-7} 
$C_3$                      &                       &  &  &  & \multicolumn{1}{l|}{} & \multicolumn{1}{l|}{} \\ \cline{4-7} 
\end{tabular}
    \caption{"Triangular" collusion pattern}
  \end{subfigure}
  \begin{subfigure}[t]{.45\linewidth}
  \centering
\begin{tabular}{ll|ll|l|l|l}
\cline{2-2} \cline{5-7}
\multicolumn{1}{l|}{$C_1$} &  &                       &  &  &  & \multicolumn{1}{l|}{} \\ \cline{2-7} 
$C_2$                    &  & \multicolumn{1}{l|}{} &  &  &  &                       \\ \cline{2-7} 
\multicolumn{1}{l|}{$C_3$} &  & \multicolumn{1}{l|}{} &  &  &  & \multicolumn{1}{l|}{} \\ \cline{2-2} \cline{4-5} \cline{7-7} 
\end{tabular}
\caption{"Hanging" collusion pattern}
  \end{subfigure}
\vspace{5 pt}
\caption{\label{fig:table-staircases}Examples of $t$-collusion equivalent patterns on $6$ elements.}
\end{figure}

\begin{proof}
\begin{itemize}
\item 1$\implies$ $t$-equivalent : If $C_1, C_2, \cdots, C_{n-t+1}$ is a triangular sequence of generators of $\mathcal{T}$, then the corresponding circuit vectors form a triangular matrix of rank $n-t-1$, and so they generate the dual matroid of $M$. Thus, every linear condition on any representation of $M$ is observed by $\mathcal{T}$, so $\mathcal{T}$ is $t$-collusion equivalent.

\item 1 $\implies$ 2 : If,  for a circuit $C_i$, there exists a collection 
 \[\{C_{i_1},C_{i_2},\cdots,C_{i_m}\}\] 
 such that
\[C_i\subseteq \bigcup_{j=1}^m C_{i_j},\]
then $i_j <i$ for all $j$. Then $C_{\min i_j}$ contains an element $e_{\min i_j}$ that is not in $C_i$ or any other $C_{i_j}$.
\item 2 $\implies$ 1 : We denote the set of generators of $\mathcal{T}$ by 
$\mathcal{S}$, and want to show that there is a triangular ordering of $\mathcal{S}$. 
By the hanging property, for any subcollection $\mathcal{S}'\subseteq\mathcal{S}$, there is at least one set in $\mathcal{S}'$ that is not covered by the others. Letting $\mathcal{S}_0= \mathcal{S}$, we can thus inductively select $C_{i+1}\in \mathcal{S}_i$ as a set that is not covered by the other sets in $\mathcal{S}_i$, and let $\mathcal{S}_{i+1}=\mathcal{S}_i -\{C_{i+1}\}$. By construction, for every $i=1...n-t+1$ there is an element $e_{i+1}\in C_{i+1}$ that is not contained in any of the sets in $\mathcal{S}_{i+1}$, so \[\{e_1,e_2,\cdots,e_{i-1}\}\cap C_i = \emptyset \quad \forall i >1.\] Thus, this is a triangular ordering of the generators in $\mathcal{T}$.
\end{itemize}
\end{proof}

Even though the two conditions in Theorem~\ref{theorem: triangular matrix conditions for basis} are equivalent, the second one is more useful in practice because it gives a direct way to test an arbitrary collusion pattern in at most $n-t+1$ steps as opposed to considering $n!$ possible permutations of the ground set. 

Moreover, Theorem~\ref{theorem: triangular matrix conditions for basis} shows that in Example~\ref{motivating example} it is sufficient to consider only $n-t+1$ of the $n$ cyclically ordered circuits.

\begin{question}
Is there an $\mathbb{F}$-representable matroid $M$ of rank $t-1$ with a set of circuits $\mathcal{S}=\{C_1,C_2,\cdots,C_{n-t+1}\}$ that does not satisfy the characterisation from Theorem~\ref{theorem: triangular matrix conditions for basis} but is a basis in every derived matroid of $M$?
\end{question}

\begin{question}
What is the minimum number $N$ of colluding circuits that guarantee that $Q=Q^\mathcal{T}$? In general, 
\[N=1+\max\{|\mathcal{H}| : \mathcal{H} - \text{hyperplane in } \delta M(Q) \}.\]
On the other hand, it has been shown that the generalized Hamming weight can be used to characterize the minimum size of a message that guarantees that the eavesdropper is able to recover the secret~\cite[Appendix]{wei1991generalized} and \cite[Section 3]{britz2007higher}. Is it possible to characterize the derived matroid of a code using its generalized Hamming distance? 
\end{question}

We will close by considering one more situation. Assume that one element $e$ of the ground set is compromised, that is, all circuits containing $e$ are observed. 

\begin{theorem}\label{theorem: compromised coordinate}
Let $M$ be a connected $\mathbb{F}$-representable matroid of rank $t-1$ on $[n]$. Fix some $e \in [n]$. Let the collusion pattern $\mathcal{T}$ be given by
\[\mathcal{T}=\langle C: e \in C\rangle.\]
Then $\mathcal{T}$ is $t$-collusion equivalent for all representations of $M$.
\end{theorem}

\begin{proof}
Denote the set of generators by $\mathcal{S}$ and let $Q$ be an arbitrary representation of $M$ over $\mathbb{F}$. We claim that $M(Q^\mathcal{T})$ is connected. On the one hand, since every circuit in $M(Q^\mathcal{T})$ is also a circuit in $M$, and $M$ is connected, the matroid $M(Q^\mathcal{T})$ does not contain any loops. On the other hand, assume there exists an element $f\in [n]$ that is a coloop in $M(Q^\mathcal{T})$. Since $M$ is connected, there is a circuit $C_{e,f}$ containing $e$ and $f$. By construction, $C_{e,f} \in \mathcal{S}$, so $C_{e,f}$ is a circuit in $M(Q^\mathcal{T})$ containing $f$.

Since $M(Q^\mathcal{T})$ is connected and contains $\mathcal{S}$, then we can apply~\cite[Theorem 4.3.2]{oxley2006matroid} to see that all circuits of $M(Q^\mathcal{T})$ not containing $e$ are given by
\begin{equation}\label{eq: avoiding circuits}
(C_1\cup C_2)-\bigcap\{C: C \in \mathcal{S}, C \subseteq C_1\cup C_2\},
\end{equation}
where $C_1$ and $C_2$ are two distinct members of $\mathcal{S}$.
However, $M$ is also connected and contains $\mathcal{S}$, so its remaining circuits $\mathcal{C}-\mathcal{S}$ are also given by~(\ref{eq: avoiding circuits}). Therefore,  $M(Q^\mathcal{T})=M$ and, thus, $Q=Q^\mathcal{T}$.
\end{proof}

\noindent \textbf{Acknowledgements.} We thank Camilla Hollanti for advising during this project. We thank Amy Wiebe, Felipe Rinc\'{o}n, Mateusz Michalek and Federico Ardila for helpful discussions about matroids. Kuznetsova was partially supported by the Academy of Finland Grant 323416 and Aalto Free Mover Scholarship.

\bibliographystyle{elsarticle-num} 
\bibliography{references}

\end{document}